\documentclass[11pt]{article}

\usepackage{amsthm, amsmath, amssymb, amsfonts, url, booktabs, tikz, setspace, fancyhdr, bm}
\usepackage{fullpage}
\usepackage{hyperref, enumerate}
\usepackage[shortlabels]{enumitem}
\usepackage[babel]{microtype}
\usepackage[english]{babel}
\usepackage[capitalise]{cleveref}
\usepackage{comment}
\usepackage{bbm}
\usepackage{csquotes}
\usepackage{graphicx}
\usepackage{float}
\usepackage[dvipsnames]{xcolor}
\usepackage{soul}
\usepackage{mathtools}
\usepackage[normalem]{ulem}

\usetikzlibrary{
    positioning,
    arrows.meta,
    shapes.geometric,
    decorations.pathmorphing,
    decorations.pathreplacing,
    fit,
    backgrounds
}

\counterwithin{figure}{section}

\newtheorem{theorem}{Theorem}[section]

\newtheorem{conj}[theorem]{Conjecture}

\newtheorem{lemma}[theorem]{Lemma}

\newtheorem{claim}[theorem]{Claim}

\theoremstyle{definition}

\newtheorem*{defn-non}{Definition}

\newlist{Case}{enumerate}{2}
\setlist[Case,1]{%
    label={\bfseries Case \arabic*.},
    labelindent=1em,
    labelwidth=1.3cm,
    labelsep*=1em,
    leftmargin=!
}
\setlist[Case,2]{%
    label={\bfseries Subcase \arabic{Casei}.\arabic*.},
    labelindent=-1em,
    labelwidth=1.3cm,
    labelsep*=1em,
    leftmargin=!
}

\newcommand{\rhat}{\widehat R}

\crefname{theorem}{Theorem}{Theorems}
\Crefname{theorem}{Theorem}{Theorems}
\crefname{lemma}{Lemma}{Lemmas}
\Crefname{lemma}{Lemma}{Lemmas}

\hypersetup{
  pdftitle={The Multicolour Size--Ramsey Number of an Even Cycle},
  pdfsubject={Multicolour size--Ramsey numbers of even cycles},
  pdfkeywords={size--Ramsey number, multicolour Ramsey theory, even cycle, random graph, vertex expansion}
}

\title{The Multicolour Size--Ramsey Number of an Even Cycle}

\author{
Lanchao Wang\thanks{School of Mathematics, Nanjing University, Nanjing, China, and ECOPRO, Institute for Basic Science, Daejeon, Korea. Email: lanchaowang@foxmail.com.}
\and
Xiaolin Wang\thanks{School of Mathematics and Statistics, Fuzhou University, Fuzhou, China. Email: xiaolinw92@163.com.}
}

\date{}

\begin{document}

\maketitle

\begin{abstract}
We determine the $k$-colour size--Ramsey number of even cycles up to
absolute constant factors. For every $k\ge2$ and every even
$n\ge200\log k$,
$$
\widehat R_k(C_n)=\Theta(k^2\log k)n.
$$
The lower bound follows from the corresponding result of Beke, Li and Sahasrabudhe for paths. For the upper bound, we combine local expansion in a random bipartite graph with BFS localisation and a path-length adjustment argument, improving the previous best estimate $O(k^{34}n)$ of Javadi, Kohayakawa and Miralaei.

\end{abstract}


 
\section{Introduction}

The notion of size--Ramsey number was introduced by Erd\H{o}s, Faudree,
Rousseau and Schelp~\cite{EFRS} and has since been extensively studied.
For a graph $H$, its $k$-colour size--Ramsey number, denoted by
$\widehat R_k(H)$, is the minimum number of edges in a graph $G$ such that
every $k$-edge-colouring of $G$ contains a monochromatic copy of $H$. We write $P_n$ and $C_n$ for the path and the cycle on $n$ vertices,
respectively.
Paths and cycles are among the most fundamental objects in size--Ramsey theory.

A landmark result of
Beck~\cite{Beck} states that $\widehat R_2(P_n)=\Theta(n)$, settling a
question of Erd\H{o}s. His argument also implies
$\widehat R_k(P_n)=\Theta_k(n)$. The dependence on the
number of colours was subsequently studied by Dudek and
Pra{\l}at~\cite{DP}, who established a quadratic lower bound, and by
Krivelevich~\cite{KrivelevichLongCycles}, who proved
$\widehat R_k(P_n)=O(k^2\log k)n$. A simpler proof with an improved explicit
constant was later given by Dudek and Pra{\l}at~\cite{DPnote}. Recently,
Beke, Li and Sahasrabudhe~\cite{BLS} closed the remaining logarithmic gap
by improving the lower bound, revealing the somewhat surprising asymptotic
order
$
\widehat R_k(P_n)=\Theta(k^2\log k)n
$
whenever $n\ge200\log k$.

For cycles, Haxell, Kohayakawa and {\L}uczak~\cite{HKL} proved
$\widehat R_k(C_n)=\Theta_k(n)$. Their argument relies
on the regularity lemma and therefore yields a tower-type dependence on $k$. Javadi, Khoeini, Omidi and
Pokrovskiy~\cite{JKOP} subsequently gave a regularity-free proof and
obtained the first bounds with non-tower-type dependence on $k$.

The dependence exhibits a striking difference between odd and even cycles. For odd cycles, Javadi and
Miralaei~\cite{JM} established an exponential lower bound in $k$, and
Brada\v{c}, Dragani\'c and Sudakov~\cite{BDS} subsequently proved a
matching upper bound up to a constant in the exponent. Thus, for
sufficiently long odd cycles,
$
\widehat R_k(C_n)=e^{\Theta(k)}n.
$

For even cycles, Javadi and Miralaei~\cite{JM} proved
$
\Omega(k^2n)
\le
\widehat R_k(C_n)
\le
O\bigl(k^{120}(\log k)^2n\bigr).
$
Brada\v{c}, Dragani\'c and Sudakov~\cite{BDS} subsequently obtained an
induced size--Ramsey bound which, in particular, improved the ordinary
upper bound to $O(k^{102}n)$. More recently, Javadi, Kohayakawa and
Miralaei~\cite{JKM} further improved the bound to $O(k^{34}n)$
via their results on long subdivisions.

On the other hand, the recent result of Beke, Li and
Sahasrabudhe~\cite{BLS} immediately yields a stronger lower bound. Indeed,
since $P_n\subseteq C_n$,
$
\widehat R_k(C_n)
\ge
\widehat R_k(P_n)
=
\Omega(k^2\log k)n.
$
Thus, prior to the present work, the best known bounds for sufficiently
long even cycles were
\[
\Omega(k^2\log k)n
\le
\widehat R_k(C_n)
\le
O(k^{34}n).
\]
Our main result closes this gap up to absolute constant factors.

\begin{theorem}\label{thm:main}
For every integer $k\ge2$ and every even integer
$n\ge 200\log k$,
$$\rhat_k(C_n)=\Theta(k^2\log k)n.
$$
\end{theorem}

Although it represents a substantial improvement over the previous
bounds, the proof uses only fairly standard tools.  We combine local expansion with a breadth-first-search (BFS)
decomposition and a path-length adjustment argument. Starting from a sparse
bipartite host, we find a monochromatic subgraph of large minimum degree and
root a BFS tree in it. By stopping at the first BFS ball
whose growth slows down, we find, within two consecutive levels at
logarithmic depth, a connected subgraph of large minimum degree. The local
expansion of the host then forces a long cycle in this subgraph. Finally, we
combine this cycle with a path-length adjustment argument and the BFS tree to obtain a range of even cycle lengths,
one of which is exactly $n$.

\smallskip
The rest of the paper is organised as follows.  In Section~\ref{sec:host},
we construct the sparse bipartite host graph, establish its local expansion,
and prove a BFS localisation lemma. Section~\ref{sec:tools} develops the tools for controlling cycle lengths
used in the argument, and Section~\ref{sec:main-proof} gives the proof
of Theorem~\ref{thm:main}.  
Finally, we conclude in
Section~\ref{sec:conclusion} with several related open problems.

\section{The host graph and BFS localisation}\label{sec:host}

We use the standard local-sparsity-to-expansion strategy for size--Ramsey
problems; see, for example, Krivelevich~\cite{KriExp}.  

For a graph $G$, we write $\overline d(G)$ for its average degree. For a graph $G$ and $X\subseteq V(G)$, let $N_G(X)$ denote the set of vertices in $V(G)\setminus X$ having a neighbour in $X$. Throughout the paper, all logarithms are natural. We omit ceiling signs for simplicity; whenever a quantity is required to be an integer, it is understood to be rounded up. 

The following lemma provides a bipartite host graph whose local sparsity guarantees the expansion.
\begin{lemma}\label{lem:host}
Let $k\ge2$ and $n\ge200\log k$, and set
$
   M=10^5kn
   $ and $
   p=(\log k)/ n.
$
Then there exists a bipartite graph $\Gamma$ with two vertex classes of
size $M$ such that:
\begin{enumerate}[(1)]
    \item\label{item:host-edges}
  $\frac12pM^2\le e(\Gamma)\le2pM^2$. In particular,
$e(\Gamma)\le2\cdot10^{10}k^2(\log k)n$.
    \item\label{item:host-local}
    $e_\Gamma(U)<10^2(\log k)|U|$ for every
    $U\subseteq V(\Gamma)$ with $1\le |U|\le3n$.

    \item\label{item:localexp} Let
$d= 10^3\log k  $. If $F\subseteq\Gamma$ has minimum degree at
least $d$, then $|N_F(X)|>2|X|$ for every nonempty
$X\subseteq V(F)$ with $|X|\le n$.
\end{enumerate}

\end{lemma}

\begin{proof}
Let $\Gamma$ be the binomial random bipartite graph with two vertex
classes of size $M$, where every cross-edge is present independently with
probability $p$.  Let
$
   \mu=\mathbb E(e(\Gamma))=pM^2.
$
By the standard Chernoff inequalities,
\[
   \Pr\!\left(
      e(\Gamma)<\frac12pM^2
      \text{ or }
      e(\Gamma)>2pM^2
   \right)
   \le2e^{-\mu/12}<\frac13.
\]

For \ref{item:host-local}, fix an
integer $u$ with $1\le u\le3n$. If
$
    10^2(\log k)u  >u^2/4,
$
then no $u$-vertex set spans $ 10^2(\log k)u  $ edges. Otherwise, we obtain
\begin{align*}
 \Pr\bigl(&\exists U\subseteq V(\Gamma),\ |U|=u,\
      e_\Gamma(U)\ge  10^2(\log k)u  \bigr)\\
 &\quad\le
   \binom{2M}{u}
   \binom{\lfloor u^2/4\rfloor}{ 10^2(\log k)u  }
   p^{ 10^2(\log k)u  }\\
 &\quad\le
   \left(\frac{2eM}{u}\right)^u
   \left(\frac{eu^2p}
   {4 10^2(\log k)u  }\right)^{
    10^2(\log k)u  }
   \qquad \
   \text{(since $\binom ab\le \left(\frac{ea}{b}\right)^b$)}\\
 &\quad\le
   \left[
      2e\cdot10^5k
      \left(\frac{e}{400}\right)^{10^2\log k}
      \left(\frac{u}{n}\right)^{10^2\log k-1}
   \right]^u
   \qquad
   \text{(by the choices of $M$ and $p$)}\\
 &\quad\le
   \left[
      \frac{2e\cdot10^5k}{3}
      \left(\frac{3e}{400}\right)^{10^2\log k}
   \right]^u
   \qquad \qquad \qquad \quad \
   \text{(since $u\le3n$)}\\
 &\quad=
   \left[
      \frac{2e\cdot10^5}{3}
      k^{\,1-10^2\log\frac{400}{3e}}
   \right]^u<
   4^{-u}. \qquad
  \qquad  \quad \ \  \  \text{(holds for every $k\ge2$)}
\end{align*}
Therefore, the probability that condition~\ref{item:host-local} fails is
less than $\sum_{u\ge1}4^{-u}=1/3$. Together with the estimate for
$e(\Gamma)$, this shows that conditions~\ref{item:host-edges} and
\ref{item:host-local} hold simultaneously with positive probability.

It remains to verify condition~\ref{item:localexp}. Suppose that $F\subseteq\Gamma$ satisfies
$\delta(F)\ge d$. If there were a nonempty set
$X\subseteq V(F)$ with $|X|\le n$ and $|N_F(X)|\le2|X|$, then, putting
$U=X\cup N_F(X)$, we would have $|U|\le3|X|\le3n$. Since every edge of
$F$ incident with a vertex of $X$ has both endpoints in $U$,
\[
   e_\Gamma(U)\ge e_F(U)
   \ge\frac12\sum_{x\in X}d_F(x)
   \ge\frac d2|X|
   \ge\frac d6|U|
   >10^2(\log k)|U|,
\]
contradicting condition~\ref{item:host-local}. Hence
$|N_F(X)|>2|X|$ for every nonempty $X\subseteq V(F)$ with $|X|\le n$,
as required.
\end{proof}

We next record a deterministic localisation lemma showing that a graph of
large minimum degree contains, within two consecutive BFS levels at
logarithmic depth, a connected subgraph of large minimum degree. For a tree $T$ with root $v$, the $j$th level of $T$ is the set
$L_j=\{x\in V(T):\operatorname{dist}_T(v,x)=j\}.
$
\begin{lemma}\label{lem:bfs-localisation}
Let $d\ge2$, and let $J$ be a connected bipartite graph with
$\delta(J)>8d$. Let $T$ be a BFS tree of $J$
rooted at an arbitrary vertex $v$, with levels
$L_0,L_1,\ldots$. Then there exist integers $q$ and $i$ such that
$$
   1\le i<q<1+\log_2|V(J)|
$$
and a connected subgraph
$
   F\subseteq J[L_i\cup L_{i+1}]
$
with
$
   \delta(F)\ge d.
$
\end{lemma}\begin{proof}
For $j\ge0$, let
$
   B_j=\bigcup_{h=0}^jL_h.
$
Let $q$ be the smallest positive integer such that
$
   |B_q|\le2|B_{q-1}|.
$
Such a $q$ exists, since eventually the BFS balls stop growing. Moreover,
since $\delta(J)>8d$,
we have $
|B_1|=1+d_J(v)>2|B_0|,
$
and hence $q\ge2$.
By the minimality of $q$,
$
   |B_{q-1}|>2^{q-1},
$
and hence
$
   q<1+\log_2|V(J)|.
$

We next show that $J[B_q]$ has average degree greater than $4d$.
Since every neighbour in $J$ of a vertex in $B_{q-1}$ lies in $B_q$, we obtain
\[
   2e\bigl(J[B_q]\bigr)
   \ge\sum_{x\in B_{q-1}}d_J(x)
   >8d|B_{q-1}|
   \ge4d|B_q|,
\]
where the last inequality follows from the choice of $q$. Therefore
$   \overline d\bigl(J[B_q]\bigr)>4d.
$

Since $J$ is bipartite, every edge of $J[B_q]$ joins two consecutive
BFS levels. Moreover,
\[
   \sum_{h=0}^{q-1}e_J(L_h,L_{h+1})
   =e\bigl(J[B_q]\bigr)
   \qquad\text{and}\qquad
   \sum_{h=0}^{q-1}\bigl(|L_h|+|L_{h+1}|\bigr)
   \le2|B_q|.
\]
Hence, for some $0\le i<q$,
\[
\begin{aligned}
   \overline d\bigl(J[L_i\cup L_{i+1}]\bigr)
   &\ge
   \frac{2\sum_{h=0}^{q-1}e_J(L_h,L_{h+1})}
        {\sum_{h=0}^{q-1}(|L_h|+|L_{h+1}|)}\ge
   \frac{e\bigl(J[B_q]\bigr)}{|B_q|}
   =\frac12\overline d\bigl(J[B_q]\bigr)
   >2d.
\end{aligned}
\]

Choose a nonempty subgraph
$
   F\subseteq J[L_i\cup L_{i+1}]
$
with the minimum number of vertices subject to
$
   \overline d(F)>2d.
$
Then $F$ is connected and $\delta(F)>d$. Finally, $i\ne0$, since
$J[L_0\cup L_1]$ is a star, while $\delta(F)> d\ge2$.
Therefore
$
   1\le i<q<1+\log_2|V(J)|,
$
as required.
\end{proof}

\section{Tools for controlling cycle lengths}\label{sec:tools}
In this section, we develop tools for controlling cycle lengths using a BFS tree.
We work with two consecutive BFS levels, so that suitable pairs of endpoints
are joined in the BFS tree by paths of the same length. Paths of varying even
lengths in the resulting two-level subgraph then yield a sequence of
consecutive even cycle lengths. This BFS-tree approach goes back to Verstra\"ete~\cite{Verstraete}; see also
Sudakov and Verstra\"ete~\cite{SV} for its use in expanding graphs, although our argument differs in the details.

The following lemma, due to Krivelevich~\cite{KrivelevichLongCycles}, shows that
expansion forces a long cycle. This standard DFS
consequence is frequently used in arguments on cycles in random graphs;
see Liu's lecture notes~\cite[Theorem~3.8]{LiuCyclesTrees}.

\begin{lemma}[{\cite{KrivelevichLongCycles}}]
\label{lem:dfs-cycle}
Let $a$ be a positive integer, let $t\ge2$, and let $G$ be a graph with
$|V(G)|>a$. Suppose that
$
   |N_G(X)|\ge t
$
for every set $X\subseteq V(G)$ satisfying $a/2\le |X|\le a$. Then $G$
contains a cycle of length at least $t+1$.
\end{lemma}

We use the following lemma of Gao, Huo and Ma~\cite[Lemma~3.2]{GHM},
whose proof builds on the Bondy--Simonovits--Verstra\"ete chorded-cycle
lemma~\cite{BondySimonovits,Verstraete} together with a path-extension
argument. It provides paths of all shorter lengths from a given cycle.

We call a partition nontrivial if both of its parts are nonempty. If
$V(G)=A\mathbin{\dot\cup}B$, an $A$--$B$ path means a path with one
endpoint in $A$ and the other in $B$.

\begin{lemma}[{\cite{GHM}}]\label{lem:detour-adjuster}
Let $G$ be a connected graph with $\delta(G)\ge3$, let
$V(G)=A\mathbin{\dot\cup}B$ be a nontrivial partition, and let $C$ be a cycle
in $G$. Unless $G$ is bipartite with bipartition $(A,B)$, for every integer
$1\le \ell<|C|$ there is an $A$--$B$ path of length $\ell$ in $G$.
\end{lemma}

The following lemma combines the above path-length lemma with the
minimal-subtree argument appearing in Verstra\"ete~\cite{Verstraete}.  For two vertices of a tree $T$, the \emph{tree path} between them means the unique
path joining them in $T$.

\begin{lemma}\label{lem:bfs-adjuster}
Let $G$ be a connected bipartite graph, let $T$ be a BFS
tree rooted at $v$, and let $L_0,L_1,\ldots$ be its levels.  Suppose that,
for some $i$, a connected subgraph
$
   F\subseteq G[L_i\cup L_{i+1}]
$
has minimum degree at least three and contains a cycle $C$.  Then there is
an integer $r$ with $1\le r\le i$ such that $G$ contains cycles of every
even length
\[
   2r+2,\ 2r+4,\ \ldots,\ 2r+|C|-2.
\]
\end{lemma}
\begin{proof}
The idea is to use the BFS tree to define a nontrivial partition
$V(F)=A\mathbin{\dot\cup}B$ by separating one branch below the deepest
common ancestor of $V(F)\cap L_i$ from the rest. 

Put
$
   U=V(F)\cap L_i.
$
Since $\delta(F)\ge3$, the set $U$ contains at least two vertices.
Let $T_U$ be the unique minimal subtree of $T$ containing every vertex of
$U$, and let $z$ be its unique vertex of minimum depth. Equivalently, $z$
is the deepest common ancestor of all vertices of $U$ in the rooted tree
$T$. Observe that $z$ lies in a level strictly smaller than $i$.

We call the components of $T-z$ rooted at children of $z$ the
\emph{child-components} of $z$. Then at least two child-components contain
vertices of $U$; otherwise all of $U$ would lie below one child of $z$,
contradicting the definition of $z$. Choose one such component and call it $K$. Define
$
   A=U\cap V(K)
$
and
$
   B=V(F)\setminus A.
$
Then both $A$ and $U\setminus A$ are nonempty, and $(A,B)$ is a nontrivial
partition of $V(F)$.

\begin{claim}\label{claim:AB-paths}
For every even integer $2\le s\le |C|-2$, there is an $A$--$B$ path
$Q_s$ of length $s$ in $F$.
\end{claim}

\begin{proof}
Since $G$ is bipartite and $T$ is a BFS spanning tree of
$G$, every edge of $G[L_i\cup L_{i+1}]$ joins $L_i$ to $L_{i+1}$. Hence $F$ has bipartition
$
   \bigl(U,V(F)\cap L_{i+1}\bigr).
$
The partition $(A,B)$ is not a bipartition of $F$: if
$u\in U\setminus A$, then any neighbour
$w\in V(F)\cap L_{i+1}$ of $u$ satisfies $u,w\in B$.
The claim now follows from Lemma~\ref{lem:detour-adjuster}.
\end{proof}

\begin{claim}\label{claim:bfs-endpoints}
For every even integer $2\le s\le |C|-2$, the endpoints of $Q_s$ lie in
two distinct child-components of $T-z$, and the unique tree path between
them has length
$
   2\bigl(i-\operatorname{dist}_T(v,z)\bigr).
$
\end{claim}

\begin{proof}
One endpoint of $Q_s$ lies in $A\subseteq U$, while the other lies in $B$.
Since $F$ is bipartite and $s$ is even, the two endpoints lie in the same
bipartition class. Hence the endpoint in $B$ also lies in $U$, and therefore
belongs to $U\setminus A$.
Thus one endpoint lies in $K$, while the other lies in a different
child-component of $T-z$. Hence their unique tree path passes through $z$.
Since both endpoints lie in $L_i$, each is at distance
$
   i-\operatorname{dist}_T(v,z)
$
from $z$, and hence their tree path has length
$
   2\bigl(i-\operatorname{dist}_T(v,z)\bigr).
$
\end{proof}

For each even $s$ with $2\le s\le |C|-2$, let $P_s$ be the tree path
between the endpoints of $Q_s$. By Claim~\ref{claim:bfs-endpoints}, the path $P_s$ has length
$2\bigl(i-\operatorname{dist}_T(v,z)\bigr)$.
Moreover, all internal vertices of $P_s$ lie in
$L_0\cup\cdots\cup L_{i-1}$, while $Q_s$ lies in
$L_i\cup L_{i+1}$. Thus $P_s$ and $Q_s$ are internally vertex-disjoint,
and their union is a cycle of length
$   2\bigl(i-\operatorname{dist}_T(v,z)\bigr)+s.
$

Letting $s$ run through $2,4,\ldots,|C|-2$ gives cycles of every even
length from
$2\bigl(i-\operatorname{dist}_T(v,z)\bigr)+2$ to
$2\bigl(i-\operatorname{dist}_T(v,z)\bigr)+|C|-2$.
Setting $r=i-\operatorname{dist}_T(v,z)$, we have $1\le r\le i$, and the
result follows.
\end{proof}

\section{Proof of Theorem \ref{thm:main}}\label{sec:main-proof}
The lower bound follows from~\cite{BLS}, since $P_n\subseteq C_n$. It remains to prove
the upper bound.
Let $\Gamma$ be a graph given by Lemma~\ref{lem:host}, and consider an
arbitrary $k$-edge-colouring of $\Gamma$. Let $H_0$ be the spanning
subgraph of $\Gamma$ formed by the edges of a colour appearing most
frequently. Let $d= 10^3\log k  $. Since $H_0$ has $2M$ vertices, condition~\ref{item:host-edges} gives
\[
   \overline{d}(H_0)
   =\frac{e(H_0)}{M}
   \ge\frac{e(\Gamma)}{kM}
   \ge\frac{pM}{2k}
   =5\cdot10^4\log k
   \ge16d.
\]

\begin{claim}\label{claim:core}
There is a connected induced subgraph $J\subseteq H_0$ such that
$\overline d(J)\ge16d$ and $\delta(J)>8d$.
\end{claim}

\begin{proof}
Choose a nonempty set $W\subseteq V(H_0)$ of minimum size such that
$J=H_0[W]$ has average degree at least $16d$. The minimality of $W$ implies that $J$ is connected and
$\delta(J)>8d$.
\end{proof}

Fix an arbitrary vertex $v\in V(J)$. Root a BFS tree
$T$ of $J$ at $v$, and let $L_0,L_1,\ldots$ be its levels.

\begin{claim}\label{claim:two-level-cycle}
For some integer
$
   1\le i<1+ \log_2|V(J)|,
$
there is a connected subgraph
$
   F\subseteq J[L_i\cup L_{i+1}]
$
with $\delta(F)\ge d$ which contains an even cycle $C$ of length at least
$n+2$.
\end{claim}

\begin{proof}
By Lemma~\ref{lem:bfs-localisation}, there exists an integer $i$ such
that
$
   1\le i<1+\log_2|V(J)|
$
and a connected subgraph
$
   F\subseteq J[L_i\cup L_{i+1}]
$
with $\delta(F)\ge d$.

By condition~\ref{item:localexp}, we have $|N_F(X)|>2|X|$ for every
nonempty $X\subseteq V(F)$ with $|X|\le n$. In particular,
$|V(F)|>n$, since otherwise taking $X=V(F)$ gives
$N_F(X)=\varnothing$.
If $n/2\le|X|\le n$, then $|N_F(X)|>2|X|\ge n$, and hence
$|N_F(X)|\ge n+1$. Applying Lemma~\ref{lem:dfs-cycle} to $F$ with $a=n$
and $t=n+1$, we obtain a cycle $C$ with $|C|\ge n+2$. Since $F$ is
bipartite, $C$ is even.
\end{proof}

Fix $i$, $F$, and $C$ as in Claim~\ref{claim:two-level-cycle}. Apply
Lemma~\ref{lem:bfs-adjuster} to $J$, with the subgraph $F$, the cycle $C$,
and the BFS tree $T$. Thus, for some integer $1\le r\le i$, the graph $J$
contains cycles of every even length from $2r+2$ to $2r+|C|-2$. Claim~\ref{claim:two-level-cycle} gives
\[
\begin{aligned}
   2r+2
   &\le2i+2
   <4+2\log_2|V(J)|
\le4+2\log_2(2M)=4+2\log_2(2\cdot10^5kn)
   <n,
\end{aligned}
\]
where the last inequality follows from a straightforward calculation using
$n\ge200\log k$. On the other hand,
$
   2r+|C|-2\ge2r+n>n.
$ Since $n$ is even, it follows that
$n\in\{2r+2,2r+4,\ldots,2r+|C|-2\}$. Hence $J\subseteq H_0$ contains a
copy of $C_n$.

Finally, condition~\ref{item:host-edges} gives
$\rhat_k(C_n)\le e(\Gamma)\le2pM^2=2\cdot10^{10}k^2(\log k)n$, completing the proof.
\qed

\section{Concluding remarks}\label{sec:conclusion}

We conclude with several related problems.
For a graph $H$, the $k$-colour induced size--Ramsey number
$\widehat R_{\mathrm{ind},k}(H)$ is the minimum number of edges
in a graph $G$ such that every $k$-edge-colouring of $G$
contains a monochromatic copy of $H$ which is induced in $G$.
The authors~\cite{WW} proved that, for all sufficiently large
odd $n$,
$
    \widehat R_{\mathrm{ind},k}(C_n)=e^{\Theta(k)}n,
$
resolving a conjecture of Brada\v{c}, Dragani\'c and
Sudakov~\cite{BDS}. Thus, for odd cycles, the ordinary and
induced size--Ramsey numbers exhibit the same exponential
dependence on $k$. 

For paths, however, the situation is rather different.
In a paper in preparation, the authors  and Yu show that
$
    \widehat R_{\mathrm{ind},k}(P_n)
    =k^2(\log k)^{2+o_k(1)}n
$
as $k\to\infty$, for all large $n$. Together with the result of Beke, Li and
Sahasrabudhe~\cite{BLS} that
$\widehat R_k(P_n)=\Theta(k^2\log k)n$,
this reveals a logarithmic separation between the ordinary
and induced size--Ramsey numbers, disproving a conjecture
of Gir\~ao and Hurley~\cite{GH} already for paths. To shed light on this difference, we briefly outline the main
ideas behind our proof.

For the upper bound, we construct a long monochromatic induced
path in any $k$-edge-colouring of a random graph by joining long
nonbacktracking walks with short connecting paths. At each step, we delete
vertices that may create unwanted edges. A key ingredient is a global bound on the number
of edges lost throughout this process, which allows us to
continue the construction until the desired length is reached.
For the lower bound, we build on the approach of Beke, Li and
Sahasrabudhe~\cite{BLS}. The main new idea is to decompose
a bipartite graph into dense pairs by means of an entropy
argument. Since a long induced path necessarily contains
two large sets with no edges between them, we develop a
greedy counting argument to show that such paths are unlikely
to appear in randomly chosen vertex subsets. This allows us
to colour many more edges at each step without creating
long monochromatic induced paths.

 We conjecture that the induced size--Ramsey numbers of paths
and even cycles have the same asymptotic order.

\begin{conj}\label{conj:induced-even-cycle}
For every $k\ge2$ and every sufficiently large even integer $n$, we have
$$
    \widehat R_{\mathrm{ind},k}(C_n)
    =k^2(\log k)^{2+o_k(1)}n.
$$
\end{conj}
The main difficulty is that a
monochromatic cycle may have chords in the host graph.
As a first step towards this conjecture,   one possible
approach is to use a somewhat denser random host, with about
$k^{5+o(1)}n$ edges, which  has large girth and
a monochromatic core with robust expansion. One could then
recursively construct a Liu--Montgomery adjuster~\cite{LM},
choosing each new piece outside the neighbourhood of the
previously constructed part so that inducedness is preserved.
Finally, one joins the two ends by a long connecting path and
uses the adjuster to vary the resulting cycle length until it
is exactly $n$.

\medskip
It would also be interesting to understand whether the
dependence on $k$ obtained in this paper extends beyond
cycles. Javadi, Kohayakawa and Miralaei~\cite{JKM}
obtained polynomial bounds in $k$ for long even
subdivisions of bounded-degree graphs.

\begin{conj}
For every fixed graph $H$ with at least one edge, any even
subdivision $H'$ of $H$ with each edge of $H$ replaced by
a path of length $\Omega_H(\log k)$ satisfies
$$
    \widehat R_k(H')=\Theta_H(k^2\log k)\,|H'|.
$$
\end{conj}
\section*{Acknowledgements}
Lanchao Wang was supported by the National Key R\&D Program of China under
grant number 2024YFA1013900, the NSFC
under grant number 12471327,  the China Scholarship Council, and the
Institute for Basic Science (IBS-R029-C4). Xiaolin Wang was supported by  the NSFC under grant number 12401447.   

At the beginning of this project, we asked ChatGPT 5.6 to work on the problems
directly. The initial attempts mainly consisted of trying to optimize methods
from previous work and did not lead to the desired bounds. The authors subsequently suggested combining
random-graph expansion for finding cycles with BFS-based methods for adjusting
cycle lengths. With this strategy in place, ChatGPT quickly produced an
argument yielding the sharp bound.  The authors then substantially revised
this argument. All mathematical arguments and proofs in the final manuscript
were written and verified by the authors. The authors assume responsibility for all content.

% The bibliography file should contain, for example,
% @article{GH,
%   author  = {Ant\'onio Gir\~ao and Eoin Hurley},
%   title   = {Embedding induced trees in sparse expanding graphs},
%   journal = {arXiv preprint arXiv:2406.04260},
%   year    = {2024}
% }
\bibliographystyle{abbrv}
\bibliography{reference}

\end{document}